%% file: skein-in-dim-3.tex
\documentclass[10pt]{scrartcl}

\usepackage{tikz}
\usetikzlibrary{matrix,arrows,decorations.pathmorphing,shapes.geometric}

\usepackage{etex}
\usepackage{subcaption}
\usepackage[utf8]{inputenc}
\usepackage{a4wide}
\usepackage{graphicx}
\usepackage{wrapfig}
\usepackage[margin=2.5cm]{geometry}

\usepackage{amsmath,accents}
\usepackage{amsthm}
\usepackage{amssymb}

\usepackage{mathrsfs, mathtools}
\usepackage{overpic}

\usepackage[utf8]{inputenc}

\usepackage{tikzit}
\input{diagrams.tikzstyles}
\usetikzlibrary{decorations.pathreplacing,decorations.markings}
\tikzset{
  on each segment/.style={
    decorate,
    decoration={
      show path construction,
      moveto code={},
      lineto code={
        \path [#1]
        (\tikzinputsegmentfirst) -- (\tikzinputsegmentlast);
      },
      curveto code={
        \path [#1] (\tikzinputsegmentfirst)
        .. controls
        (\tikzinputsegmentsupporta) and (\tikzinputsegmentsupportb)
        ..
        (\tikzinputsegmentlast);
      },
      closepath code={
        \path [#1]
        (\tikzinputsegmentfirst) -- (\tikzinputsegmentlast);
      },
    },
  },
  mid arrow/.style={postaction={decorate,decoration={
        markings,
        mark=at position .5 with {\arrow[#1]{stealth}}
      }}},
}
\tikzset{%
  link/.style    = { white, double = blue, line width =2.0pt,
                     double distance = 1.1pt },
    link2/.style    = { white, double = black, line width = 1.8pt,
  	double distance = 0.4pt },
  channel/.style = { white, double = blue, line width = 0.8pt,
                     double distance = 0.6pt },
}

\usepackage[all,cmtip]{xy}

\makeatletter
\DeclareRobustCommand{\em}{%
	\@nomath\em \if b\expandafter\@car\f@series\@nil
	\normalfont \else \slshape \fi}
\makeatother

\usepackage[hidelinks]{hyperref}

\numberwithin{equation}{section}
\usepackage{mathtools}
\mathtoolsset{showonlyrefs}
\numberwithin{equation}{section}

\newtheoremstyle{style1}
{13pt}
{13pt}
{}
{}
{\normalfont\bfseries}
{.}
{.5em}
{}

\theoremstyle{style1}

\newtheorem{definition}{Definition}[section]

\newtheorem{remark}[definition]{Remark}

\makeatletter
\newtheorem*{repd@theorem}{\repd@title}
\newcommand{\newrepdtheorem}[2]{%
	\newenvironment{repd#1}[1]{%
		\def\repd@title{#2 \ref{##1}}%
		\begin{repd@theorem}}%
		{\end{repd@theorem}}}
\makeatother

\newrepdtheorem{definition}{Definition}

\newcommand{\catf}[1]{{\mathsf{#1}}}

\newtheoremstyle{style2}
{13pt}
{13pt}
{\slshape}
{}
{\normalfont\bfseries}
{.}
{.5em}
{}

\theoremstyle{style2}

\makeatletter
\newtheorem*{rep@theorem}{\rep@title}
\newcommand{\newreptheorem}[2]{%
	\newenvironment{rep#1}[1]{%
		\def\rep@title{#2 \ref{##1}}%
		\begin{rep@theorem}}%
		{\end{rep@theorem}}}
\makeatother

\newreptheorem{theorem}{Theorem}
\newreptheorem{corollary}{Corollary}

\newtheorem{theorem}[definition]{Theorem}
\newtheorem{proposition}[definition]{Proposition}
\newtheorem{corollary}[definition]{Corollary}

\usepackage{tikz}
\usetikzlibrary{matrix,arrows,decorations.pathmorphing,shapes.geometric}
\usepackage{tikz-cd}

\usepackage{enumitem}

\usepackage{needspace}
\newcommand{\spaceplease}{\needspace{5\baselineskip}}

\newcommand{\Ca}{\mathcal{A}}

\newcommand{\Map}{\catf{Map}}
\newcommand{\Diff}{\catf{Diff}}

\newcommand{\ra}[1]{\xrightarrow{\ #1 \ }}

\newcommand{\colimsub}[1]{\underset{#1}{\operatorname{colim}}\,}

\newcommand{\Hbdy}{\catf{Hbdy}}

\usepackage{pifont}

\newcommand{\cat}[1]{\mathcal{#1}}

\newcommand{\End}{\catf{End}}

\newcommand{\Hom}{\operatorname{Hom}}

\newcommand{\id}{\operatorname{id}}

\newcommand{\vect}{\catf{vect}}

\let\to\undefined
\newcommand{\to}{\longrightarrow}
\let\mapsto\undefined
\newcommand{\mapsto}{\longmapsto}

\newcommand{\Rexf}{\catf{Rex}^{\mathsf{f}}}
\newcommand{\Lexf}{\catf{Lex}^\mathsf{f}}

\newcommand{\Vect}{\catf{Vect}}

\newcommand{\ARGraphs}{\cat{A}\catf{-RibGraphs}}

\newcommand{\skA}{\catf{sk}_\cat{A}}
\newcommand{\SkA}{\catf{Sk}_\cat{A}}
\newcommand{\SkCatA}{\catf{SkCat}_\cat{A}}
\newcommand{\skcatA}{\catf{skcat}_\cat{A}}

\newcommand{\opp}{\text{opp}}

\let\colon\undefined\newcommand{\colon}{:}

	\newcommand{\Proj}{\catf{Proj}\,}
\DeclareMathSymbol{\Phiit}{\mathalpha}{letters}{"08} 
\DeclareMathSymbol{\Psiit}{\mathalpha}{letters}{"09}
\DeclareMathSymbol{\Sigmait}{\mathalpha}{letters}{"06}
\DeclareMathSymbol{\Xiit}{\mathalpha}{letters}{"04}
\DeclareMathSymbol{\Piit}{\mathalpha}{letters}{"05}\let\Pi\undefined\newcommand{\Pi}{\Piit}
\DeclareMathSymbol{\Gammait}{\mathalpha}{letters}{"00}
\DeclareMathSymbol{\Omegait}{\mathalpha}{letters}{"0A}\let\Omega\undefined\newcommand{\Omega}{\Omegait}
\DeclareMathSymbol{\Upsilonit}{\mathalpha}{letters}{"07}
\DeclareMathSymbol{\Thetait}{\mathalpha}{letters}{"02}
\DeclareMathSymbol{\Lambdait}{\mathalpha}{letters}{"03}\let\Lambda\undefined\newcommand{\Lambda}{\Lambdait}

\let\Phi\undefined\newcommand{\Phi}{\Phiit}
\let\Sigma\undefined\newcommand{\Sigma}{\Sigmait}
\let\Psi\undefined\newcommand{\Psi}{\Psiit}
\let\Gamma\undefined\newcommand{\Gamma}{\Gammait}
\newcommand{\nakar}{\catf{N}^\catf{r}}
\newcommand{\nakal}{\catf{N}^\catf{l}}
\newcommand{\Sk}{\catf{Sk}}
\newcommand{\SkAlg}{\catf{SkAlg}}
\newcommand{\PhiA}{\Phi_{\! \cat{A}}}

\usepackage{multicol}
\newenvironment{pnum}{\begin{enumerate}[label=(\roman*)]}{\end{enumerate}}

\setkomafont{section}{\rmfamily\large}
\setkomafont{subsection}{\rmfamily}
\setkomafont{subsubsection}{\rmfamily}
\setkomafont{subparagraph}{\rmfamily} 

\makeatletter
\renewcommand\section{\@startsection {section}{1}{\z@}%
	{-3.5ex \@plus -1ex \@minus -.2ex}%
	{2.3ex \@plus.2ex}%
	{\normalfont\scshape\centering}}
\makeatother

\usepackage{titletoc}

\dottedcontents{section}[1em]{\rmfamily}{1em}{0.2cm}
\dottedcontents{subsection}[0em]{}{3.3em}{1pc}

\usepackage{titlesec}
\titleformat{\subsection}[runin]
{\normalfont\bfseries}
{\thesubsection}
{0.5em}
{}
[.]

\definecolor{Blue}  {rgb} {0.282352,0.239215,0.803921}
\definecolor{Green} {rgb} {0.133333,0.545098,0.133333}
\definecolor{Red}   {rgb} {0.803921,0.000000,0.000000}
\definecolor{Violet}{rgb} {0.580392,0.000000,0.827450}

\usepackage{multicol}

\usepackage{titletoc}
\dottedcontents{section}[1em]{\rmfamily}{1.5em}{0.2cm}
\dottedcontents{subsection}[5em]{\rmfamily}{1.9em}{0.2cm}

\definecolor{Blue}  {rgb} {0.282352,0.239215,0.803921}
\definecolor{Green} {rgb} {0.133333,0.545098,0.133333}
\definecolor{Red}   {rgb} {0.803921,0.000000,0.000000}
\definecolor{Violet}{rgb} {0.580392,0.000000,0.827450}

\newcounter{jfc}

\usepackage{dingbat}

\begin{document}

	\vspace*{0.5cm}
	\begin{center}	\textbf{\large{Admissible Skein Modules and Ansular Functors:  A Comparison}}\\	\vspace{1cm}	{\large Lukas Müller $^{a}$} \ and \ \ {\large Lukas Woike $^{b}$}\\ 	\vspace{5mm}{\slshape $^a$ Perimeter Institute  \\  N2L 2Y5 Waterloo \\ Canada}	\\[7pt]	{\slshape $^b$ Université Bourgogne Europe\\ CNRS\\ IMB UMR 5584\\ F-21000 Dijon\\ France }\end{center}	\vspace{0.3cm}	
	\begin{abstract}\noindent 
		Given a finite ribbon category, there are two possibilities for an extension defined on all three-dimensional handlebodies: On the one hand, one can use the admissible skein module construction of Costantino-Geer-Patureau-Mirand. On the other hand, by a construction of the authors using Costello's modular envelope, one can build a so-called ansular functor, a handlebody version of the notion of a modular functor. Unlike the admissible skein modules with their construction through the Reshetikhin-Turaev graphical calculus, the ansular functor is defined purely through a universal property. In this note, we prove the widely held expectation that these constructions are related by giving an isomorphism between them, with the somewhat surprising subtlety that we need to include consistently on one of the sides an additional boundary component labeled by the distinguished invertible object of Etingof-Nikshych-Ostrik. Our comparison result includes the handlebody group action as well as the skein algebra action.
	\end{abstract}

\tableofcontents

\spaceplease
\section{Introduction and summary}
Skein theory~\cite{turaevck,turaevskein,hp92,Przytycki,Walker}
 is often encountered as a construction procedure for various kinds of state spaces in quantum topology that are 
  part of some sort of topological field theory.
 For a given manifold $M$ (for us  of dimension 
 three, moreover 	 oriented, possibly with boundary)
 and a certain 
 ribbon category $\cat{A}$,	 one builds the so-called \emph{skein module} $\Sk_\cat{A}(M)$ for $\cat{A}$ and $M$, a vector space.
 The exact structure and properties needed on $\cat{A}$ will depend on the context and will be discussed momentarily. 
 Often $\cat{A}$ has a representation-theoretic origin. It can for instance arise as a category of modules over a quantum group, but it might as well come from a suitable vertex operator algebra.
 The skein module $\Sk_\cat{A}(M)$ is, roughly speaking, freely generated by all 
 ribbons in $M$ that are labeled in a certain way by the objects and morphisms of $\cat{A}$, modulo local relations arising from the evaluation of the Reshetikhin-Turaev
 graphical calculus of $\cat{A}$ on 
 three-dimensional cubes. 
 The skein module $\SkAlg_\cat{A}(\Sigma)= \Sk_\cat{A}(\Sigma \times [0,1])$, where $\Sigma$ is a surface (again always compact and oriented),
 comes with an algebra structure through the stacking along the interval direction.
 One calls this the \emph{skein algebra} of $\Sigma$.
 If $\Sigma = \partial M$, the skein algebra $\SkAlg_\cat{A}(\Sigma)$ acts on the skein module $\Sk_\cat{A}(M)$.
 We refer to \cite{roberts,masbaumroberts} for more background or~\cite{masbaumja} for a short overview.

For the three-dimensional version of skein theory
discussed in this note, $\cat{A}$ will be a \emph{finite ribbon category} in the sense of \cite{etingofostrik,egno}.\label{deffiniteribbon} This means:
\begin{itemize}
	\item $\cat{A}$ is a \emph{finite $k$-linear category} over an algebraically closed field $k$ that we fix throughout. Finiteness means in this context that $\cat{A}$ is a linear abelian category with finite-dimensional morphism spaces,
	finitely many simple objects up to isomorphism, enough projective objects and finite length for every object.
	
	\item $\cat{A}$ has a monoidal product $\otimes : \cat{A}\boxtimes \cat{A}\to \cat{A}$, where $\boxtimes$ is the Deligne product, that is rigid ($\otimes$  has two-sided duals), and has a simple unit. 
	One then calls $\cat{A}$ a \emph{finite tensor category}.
	We denote the dual of $X\in \cat{A}$ by $X^\vee$.
	The remaining conditions will actually imply that this is a two-sided duality.
	Therefore, we allow ourselves not to distinguish between the two types of dualities.

	\item The monoidal product $\otimes$ comes with a \emph{braiding} $c_{X,Y}:X\otimes Y \cong Y \otimes X$ and a \emph{balancing} $\theta_X:X\to X$, i.e.\ a natural automorphism with $\theta_I=\id_I$ for the monoidal unit $I$ of $\otimes$ and $\theta_{X\otimes Y}=c_{Y,X}c_{X,Y}(\theta_X \otimes \theta_Y)$.
	This balancing is \emph{ribbon} in the sense that $\theta_{X^\vee}=\theta_X^\vee$.
	\end{itemize}
A finite ribbon category whose braiding is \emph{non-degenerate} in the sense that $c_{Y,X}c_{X,Y}=\id_{X \otimes Y}$ for all $Y\in\cat{A}$ implies that $X$ is isomorphic to a direct sum of finitely many copies of $I$ is called a \emph{modular category}~\cite{turaev,baki,kl,shimizumodular}.

The classical skein-theoretic constructions pertain to the case in which the finite ribbon category $\cat{A}$ is semisimple. A finite ribbon category that is semisimple is called a \emph{ribbon fusion category}.
If $\cat{A}$ is actually a \emph{modular fusion category}, i.e.\ modular and semisimple, then 
for any handlebody $H$ (for us always three-dimensional, compact and oriented)
the skein module $\Sk_\cat{A}(H)$ appears as the state space of an anomalous three-dimensional topological field theory, the \emph{Reshetikhin-Turaev theory}~\cite{rt1,rt2,turaev} evaluated on the boundary surface $\Sigma$ of $H$. 
In particular, the representation of the mapping class group $\Map(H)$ of the handlebody $H$ on $\Sk_\cat{A}(H)$  
extends in this case to a projective representation of the mapping class group $\Map(\Sigma)$ of the surface $\Sigma$. A skein-theoretic construction of these representations is given in~\cite{roberts,masbaumroberts}, see \cite{andersenfaithful,fww} for the remarkable asymptotic faithfulness
 properties of these representations that are often referred to as \emph{quantum representations} of mapping class groups.

 One major 
  problem that has been approached from various different angles is the generalization of skein-theoretic methods
 beyond semisimplicity,
 a vast generalization
 that is technically demanding.
 A rather conceptual approach to skein-theoretic methods is the one based on \emph{factorization homology}~\cite{AF}.
 For a finite ribbon category $\cat{A}$ (or more generally a framed $E_2$-algebra) and a surface $\Sigma$ (always compact and oriented, with parametrized boundary), factorization homology produces a category $\int_\Sigma \cat{A}$; we should think of $\int_\Sigma \cat{A}$ as the result of integrating $\cat{A}$ over the surface $\Sigma$. Applications of the factorization homology of ribbon categories were given in~\cite{bzbj,bzbj2,bjss,skeinfin}, and the  viewpoint on skein theory proposed therein is the following: Since factorization homology is functorial with respect to embeddings, the embedding $\emptyset \to \Sigma$ of the empty manifold into $\Sigma$ induces a right exact functor $\vect \to \int_\Sigma \cat{A}$ from the category of finite-dimensional vector spaces over $k$ to $\int_\Sigma \cat{A}$. This functor is determined by its value on $k\in\vect$ and gives us an object $\cat{O}_\Sigma^\cat{A}\in \int_\Sigma \cat{A}$, the so-called \emph{quantum structure sheaf}. Its endomorphism algebra
 \begin{align}
 \label{eqndefskeinalg}	\SkAlg_\cat{A}(\Sigma) = \End_{\int_\Sigma \cat{A}}(\cat{O}_\Sigma^\cat{A}) 
 	\end{align}
 is the skein algebra of $\cat{A}$ on $\Sigma$~\cite[Definition~2.6]{skeinfin}. 
If one wants to use~\eqref{eqndefskeinalg} as the definition of the skein algebra, then one must show that~\eqref{eqndefskeinalg} reduces to the classical definition of the skein algebra if $\cat{A}$ is a ribbon fusion category. This non-trivial fact was established in~\cite{cooke}.
Note that in~\eqref{eqndefskeinalg} $\cat{A}$ just needs to be a framed $E_2$-algebra, i.e.\ a balanced and braided category. It does not have to be rigid.

Skein modules for handlebodies also fit into this picture as follows. To this end, we need to fix some conventions:
We will often consider handlebodies $H$
(always three-dimensional compact oriented) equipped with a finite number of parametrized disks 
embedded into their boundary. By a slight abuse of notation we will denote the choice of a handlebody together with these embeddings just by $H$. We write $\Map(H)$ for the mapping class group of $H$, i.e.\ the 
group of isotopy classes of orientation-preserving diffeomorphisms fixing the embedded disks point-wise. Restricting an element of $\Map(H)$ to the boundary induces a group homomorphism from $\Map(H)$ to the mapping class group of the surface with boundary $\Sigma$ formed from $\partial H$ by removing the interior of the embedded disks.

If $\cat{A}$ is a finite ribbon category, then it is a particular instance of a \emph{cyclic} framed $E_2$-algebra~\cite{cyclic}, with the notion of cyclicity being a higher categorical version of the one in~\cite{gk}.
 Using the results of~\cite{costello,giansiracusa}, it is then shown in~\cite{mwansular} that $\cat{A}$ extends uniquely to a modular algebra over the modular operad of handlebodies, a so-called \emph{ansular functor} $\widehat{\cat{A}}$, the `handlebody version' of the familiar notion of a modular functor~\cite{Segal,ms89,turaev,tillmann,baki}.
Its value $\widehat{\cat{A}}(H)$ on a handlebody $H$, sometimes referred to as \emph{space of conformal blocks}, comes with a representation of the mapping class group $\Map(H)$ of the handlebody.  
For $n\ge 0$ disks embedded in $\partial H$, we obtain more generally a right exact functor $\widehat{\cat{A}}(H):\cat{A}^{\boxtimes n}\to\vect$.

The ansular functor $\widehat{\cat{A}}$ induces moreover a right exact functor $\PhiA(H):\int_\Sigma \cat{A} \to \vect$, the so-called \emph{generalized handlebody skein module}~\cite[Section~4]{brochierwoike}. Since $\PhiA(H) (\cat{O}_\Sigma^\cat{A})$ can be canonically identified with $\widehat{\cat{A}}(H)$,
the vector space
 $\widehat{\cat{A}}(H)$ comes, thanks to~\eqref{eqndefskeinalg}, 
 with an action of the skein algebra $\SkAlg_\cat{A}(\Sigma)$. 
The construction of $\widehat{\cat{A}}$ and $\PhiA$ is purely topological, and the category $\cat{A}$ does not even have to be rigid. A weaker form of duality, a so-called \emph{Grothendieck-Verdier duality} in the sense of~\cite{bd} suffices. Instead of requiring  the monoidal product to be rigid, we just ask for an  anti-equivalence $D:\cat{A}\to\cat{A}^\opp$, to be thought of as a weak form of duality, and a dualizing object $K\in \cat{A}$ (in the rigid case, this would be the unit) connected by natural isomorphisms $\cat{A}(X\otimes Y,K)\cong \cat{A}(X,DY)$ for all $X,Y\in \cat{A}$, where $\cat{A}(-,-)$ denotes the hom vector spaces in the linear category $\cat{A}$.
 One also asks for $D$ and $\theta$ to be compatible,
 $D\theta_X = \theta_{DX}$; we refer to \cite{bd,cyclic,brochierwoike} for details. We should add that the construction of $\widehat{\cat{A}}$ and $\PhiA$ works for any cyclic framed $E_2$-algebra in a symmetric monoidal bicategory.
	
	Even though the construction of $\widehat{\cat{A}}(H)$ for a handlebody $H$ does not need $\cat{A}$ to be rigid, this seems necessary if one wants to define a skein module for $\cat{A}$ on every three-dimensional manifold $M$. Such a construction was given by
	Costantino, Geer and Patureau-Mirand
	 in~\cite{asm} where the authors build, without any relation to the above-mentioned objects $\widehat{\cat{A}}$ or $\PhiA$,  a so-called \emph{admissible skein module} $\skA(M)$ by restricting the classical skein-theoretic methods to the tensor ideal of projective objects of the finite ribbon category $\cat{A}$. This skein module has an action
	of the mapping class group of the handlebody
	 and an action of a type of skein algebra for $\partial M$, but these skein algebras are not necessarily  unital.

For a ribbon category $\cat{A}$ and a handlebody $H$, there are now two a priori different constructions possible: 
We have the value $\widehat{\cat{A}}(H)$ of the ansular functor associated to $\cat{A}$ on a handlebody
$H$ with its $\Map(H)$-action~\cite{cyclic,mwansular}, which is also a module over the skein algebra~\cite[Section~4]{brochierwoike} thanks to the functors $\PhiA$.  On the other hand, we have the admissible skein module $\skA(H)$ for $\cat{A}$ and $H$~\cite{asm}.
Of course, $\widehat{\cat{A}}(H)$ can be defined without rigidity while the admissible skein modules require rigidity and can be defined on all three-dimensional manifolds. In other words, both constructions admit generalizations in somewhat orthogonal directions, but we need to compare them on their \emph{common domain of definition}. 
	This comparison will have to bring together rather abstract constructions via the modular envelope of cyclic operads and factorization homology with the construction used for admissible skein modules that is based on the graphical calculus of ribbon categories. 
	It seems to be an expectation that such a connection can be made, and it is the purpose of this note to formulate and prove such a comparison statement.

	A first guess might be that both constructions are equivalent, but it is not only difficult to devise a direct proof of this --- it is also, in this na\" ive form, false! 
	By \cite[Theorem~7.8]{cyclic}, for a rigid Grothendieck-Verdier duality, the genus zero space of conformal blocks with no insertions (this is the value for a three-dimensional ball) is given by the endomorphisms of the monoidal unit, which is just the ground field $k$. In other words, the usual normalization axiom in~\cite{baki} is satisfied in this case. 
	By \cite[Corollary~3.2]{asm} the admissible skein module
	 for the three-dimensional ball is the linear dual of the space of possibly degenerate modified traces.
	This is isomorphic to $k$ if and only if $\cat{A}$ is \emph{unimodular}, i.e.\
	if and only if
	the distinguished invertible object $\alpha \in \cat{A}$ from~\cite{eno-d}
	controlling the quadruple dual via $-^{\vee\vee\vee\vee}\cong \alpha \otimes - \otimes \alpha^{-1}$ is isomorphic to the monoidal unit.
	Otherwise, it is zero.
	
	The mistake in the na\" ive approach to the comparison
	is the assumption that the correct Grothen\-dieck-Verdier duality that we need to choose on $\cat{A}$ for the comparison is the rigid duality.
	In fact, by \cite[Theorem~4.2]{mwcenter} the possible ribbon Grothendieck-Verdier 
	dualities on $\cat{A}$ as balanced braided category are of the form $D=\beta \otimes -^\vee$,
	 where $\beta \in \cat{A}$ is in the balanced Müger center, i.e.\ it
	  trivially double braids with all objects and has trivial balancing. 
	Let us denote this ribbon Grothendieck-Verdier category with the $\beta$-shifted duality by $\cat{A}_\beta$. In the comparison statement, we will encounter such a shifted duality.
	
	There is one additional technical detail that is more cosmetic in nature that we have to address before formulating the comparison:
	Let $M$ be a three-dimensional oriented manifold with a 
	collection of 
	$n\ge 0$ disks embedded in the boundary surface $\partial M$, each of which are labeled by a projective object $P_i \in \Proj \cat{A}$, where $1\le i\le n$. The construction of~\cite{asm} then gives us an admissible skein module $\skA(M;P_1,\dots,P_n)$. Actually, it would suffice to require one of the objects per connected component to be projective, but this does not make a difference by the arguments from~\cite{brownhaioun}. 
	In our construction using the modular envelope and factorization homology, there is no restriction to projective objects, so we have to resolve this discrepancy. 
	To the linear functor $\skA(M;-)$ from~\cite{asm},
	 we apply the standard procedure
	 of finite free linear cocompletion to obtain a right exact
	 functor $\SkA(M;-):\cat{A}^{\boxtimes n} \to \vect$, the \emph{cocompleted admissible skein module}. 
	We can now state our comparison result for the ansular functor $\widehat{\cat{A}}_{\alpha^{-1}}$
	for the $\alpha^{-1}$-shifted version of $\cat{A}$ with the cocompleted admissible skein modules, restricted to handlebodies:
	
	\spaceplease
	\begin{reptheorem}{thmahatasm}[Comparison result]
			Let $\cat{A}$ be a finite ribbon category and $H$ a handlebody with $n\ge 0$ disks embedded in its boundary surface. Then there is a canonical natural $\Map(H)$-equivariant isomorphism of functors
		\begin{align}
		\widehat{\cat{A}}_{\alpha^{-1}}(H;-) \ra{\cong} \SkA(H;-) : \cat{A}^{\boxtimes n} \to \vect
		\end{align}
		between the value of the modular extension $\widehat{\cat{A}}_{\alpha^{-1}}$ of the balanced braided category 
		$\cat{A}$ with Grothendieck-Verdier duality $D=\alpha^{-1}\otimes -^\vee$  on $H$ and the cocompleted
		admissible skein module $\SkA(H;-)$. 
		In the case $n=0$ (no embedded disks), the completion step for admissible skein modules is not needed, and we obtain a $\Map(H)$-equivariant isomorphism
		\begin{align}
		\widehat{\cat{A}}_{\alpha^{-1}}(H) \ra{\cong} \skA(H) 
		\end{align}
		of vector spaces
		for the uncompleted admissible skein modules $\skA(H) $.
	\end{reptheorem}

We explain in Remark~\ref{rembcharge} why this can be interpreted as $\widehat{\cat{A}}$ and the admissible skein modules agreeing
on handlebodies
 up to a background charge.

The comparison in Theorem~\ref{thmahatasm}
is phrased as an equivalence of ansular functors in Theorem~\ref{Thmcomparisionansular}, but here in the introduction we prefer to present the spelled out version.
	The main technical ingredient for the proof of the comparison result
	are the results of~\cite{cyclic,mwansular} on cyclic algebras 
	 and the factorization homology construction of spaces of conformal blocks~\cite{brochierwoike}, in combination with the very recent results~\cite{brownhaioun,rst} on the admissible skein modules of~\cite{asm}. 
	
	Theorem~\ref{thmahatasm} is a comparison as handlebody group representations. For the skein module action, we can make the following statement:
	For a handlebody $H$ with $n\ge 0$ disks embedded in its boundary, $\Sigma = \partial H$ and a finite ribbon category $\cat{A}$ (in fact even a ribbon Grothendieck-Verdier category), we have, as explained above, the
	generalized handlebody skein module
	$\PhiA(H):\int_\Sigma \cat{A}\to \cat{A}^{\boxtimes n}$ from \cite[Section~4]{brochierwoike}
	 that turns out to be a $\int_{\partial \Sigma \times [0,1]} \cat{A}$-module map.
	On the other hand, through
	the construction in~\cite[Section~2.3]{brownhaioun}, we obtain a right exact functor $\SkA^\Phi(H):\SkCatA(\Sigma) \to \cat{A}^{\boxtimes n}$
	 (here $\SkCatA(\Sigma)$ is the finite free cocompletion of $\skcatA(\Sigma)$), and by \cite[Theorem~3.10]{brownhaioun} an equivalence
	\begin{align}  
	\int_\Sigma \cat{A}\simeq \SkCatA(\Sigma) \ ,  \label{eqnequivfhskein}
	\end{align}
	which is exactly a non-semisimple generalization of the results in~\cite{cooke}.

	\begin{reptheorem}{thmgenskeinmodule}
		For a unimodular finite ribbon category $\cat{A}$ and a handlebody $H$ with $n\ge 0$ embedded disks in its boundary surface $\Sigma$ (with as usual the interior of the disks removed), the generalized handlebody skein module
		$\PhiA(H):\int_\Sigma \cat{A} \to \cat{A}^{\boxtimes n}$ agrees, under the equivalence~\eqref{eqnequivfhskein} between factorization homology and the freely cocompleted skein category,
		with the map $\SkA^\Phi(H;-) : \SkCatA(\Sigma) \to \cat{A}^{\boxtimes n}$
		induced by the admissible skein modules as a module map over $\int_{\partial \Sigma \times [0,1]} \cat{A}$.  
	\end{reptheorem}

As mentioned above, the map $\PhiA(H):\int_\Sigma \cat{A}\to \cat{A}^{\boxtimes n}$ endows $\widehat{\cat{A}}(H)$ with an action of the skein algebra $\SkAlg_\cat{A}(\Sigma) = \End_{\int_\Sigma \cat{A}}(\cat{O}_\Sigma^\cat{A}) $.
On the admissible skein module $\SkA(H;-)$, there is also an action of the skein algebra thanks to the result~\eqref{eqnequivfhskein} from~\cite{brownhaioun}.
With Theorem~\ref{thmgenskeinmodule},
we prove:

\begin{repcorollary}{corskeinmodules}
For a unimodular finite ribbon category $\cat{A}$, and any handlebody $H$, there is an isomorphism $\widehat{\cat{A}}(H)\cong \SkA(H)$ of skein modules, i.e.\ as modules over  $\SkAlg_\cat{A}(\Sigma)$. 
\end{repcorollary}

This is a comparison of the skein module structure from \cite[Section~4]{brochierwoike} with the one from~\cite[Section~2.3]{brownhaioun}.
The comparison of the latter with the non-unital skein action of~\cite{asm} is discussed already in \cite[Remark~2.19]{brownhaioun}.
	
As another immediate consequence of Theorem~\ref{thmgenskeinmodule},
we can express the admissible skein modules of a closed three-dimensional manifold entirely through factorization homology quantities, a fact that will be relevant in
 follow-up work:
\spaceplease
\begin{repcorollary}{corexcisionphi}[Excision for admissible skein modules in terms of the generalized skein modules $\PhiA$]
	Let $\cat{A}$ be a unimodular finite ribbon category, $M$ a closed oriented three-dimensional manifold and 
	$M=H' \cup_\Sigma H$ a Heegaard splitting for $M$, with gluing along a surface $\partial \bar{H'} = \Sigma = \partial H$.
	Then the admissible skein module for $M$
	can be obtained as a coend over the generalized skein modules $\PhiA(H)$ and $\PhiA(H')$, i.e.\ there is a canonical isomorphism
	\begin{align}
	\label{eqnheegaard}	\int^{P \in \Proj \int_\Sigma \cat{A}} \PhiA(H';P^\vee)\otimes \PhiA(H;P) \ra{\cong}
		\skA(M) \ . 
	\end{align}
\end{repcorollary}

Finally, let us highlight that the comparison
of ansular functors and admissible skein modules should suggest \emph{in no way}
that the two constructions could replace each other, thereby making the other construction superfluous.
Ansular functors can be constructed beyond rigid dualities for which skein-theoretic methods using a graphical calculus remain unavailable.
On the other hand, ansular functors a priori do not allow to build skein modules for three-manifolds that are not handlebodies. 
For example, it is only through the comparison that we know that expressions like the left hand side of~\eqref{eqnheegaard} are independent of the Heegaard splitting.

Instead, the comparison is enriching for both sides: For the ansular functors, we obtain --- at least in the rigid case --- a much needed graphical calculus.
The admissible
skein modules in turn inherit a characterization through a universal property that comes through the description via the modular envelope that is now available thanks to the comparison result.
In more detail, the skein modules $\SkA(H;-)$ for handlebodies are, as handlebody group representations, determined by the fact that on balls with boundary labels $X_1,\dots,X_n\in\cat{A}$ they are given by the morphism spaces
$\cat{A}( X_1 \otimes \dots \otimes X_n ,\alpha^{-1}  )^*$ and the compatibility with gluing.
This determines $\SkA(H;-)$ as ansular functor. In the unimodular case, the generalized handlebody skein modules $\PhiA(H)$ resulting from the ansular functor then determine the skein modules on closed oriented three-manifolds as well via Corollary~\ref{corexcisionphi}.

	\vspace*{0.2cm}\textsc{Acknowledgments.} We thank
Adrien Brochier, Jennifer Brown, Benjamin Ha\"ioun and Deniz Yeral for helpful discussions related to this project.
LM gratefully acknowledges support of the Simons Collaboration on Global Categorical Symmetries. Research at Perimeter Institute is supported in part by the Government of Canada through the Department of Innovation, Science and Economic Development and by the Province of Ontario through the Ministry of Colleges and Universities. The Perimeter Institute is in the Haldimand Tract, land promised to the Six Nations.
LW gratefully acknowledges support
by the ANR project CPJ n°ANR-22-CPJ1-0001-01 at the Institut de Mathématiques de Bourgogne (IMB).
The IMB receives support from the EIPHI Graduate School (contract ANR-17-EURE-0002).

\section{A mostly self-contained, categorically oriented introduction to skein modules in dimension three}
In this section, we recall the definition and basic properties of
 the skein modules of a finite ribbon category $\cat{A}$.
We will discuss 
 the possibly non-semisimple situation, meaning that we will directly present the admissible 
skein module construction~\cite{asm}. In some places, 
we will use tools from~\cite{fsy-sn,sn,brownhaioun,rst} 
to adapt the presentation a little more to the purposes of this note.

For every oriented three-dimensional manifold $M$, 
we fix the additional structure  of a boundary parametrization, i.e.\ an
orientation-preserving embedding of a two-dimen\-sional compact oriented surface (potentially with boundary) into the boundary $\partial M$
of $M$. We call \begin{itemize}
	\item the complement of the image of the embedding in $\partial M$
the \emph{free boundary}, \item the image of closed surfaces under the embedding
the \emph{closed boundary}, \item and the image of surfaces with non-empty
boundary under the embedding the \emph{open boundary},   \end{itemize}
with the terminology being borrowed from the analogous two-dimensional situation discussed for example in~\cite[Section~2]{sn}.

A \emph{ribbon graph} in $M$~\cite[Chapter~I.2]{turaev} is a collection of embedded oriented ribbons (`thin rectangles' $[0,1]^2$ or `thin loops' $\mathbb{S}^1\times [0,1]$) and coupons (`thick' rectangles $[0,1]^2$) in $M$, see Figure~\ref{Fig:Ribbon1}.\begin{figure}[h]
	\begin{center}
		\begin{overpic}[scale=1]
			{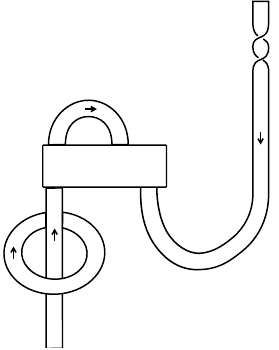}
		\end{overpic}
	\end{center}
	\caption{A local illustration of a ribbon graph in a three dimensional manifold with four ribbons and one coupon.}
	\label{Fig:Ribbon1}
\end{figure}     
 Ribbons (if they are of the form $[0,1]^2$) and coupons both have an \emph{in-boundary} and an \emph{out-boundary} given by $\{0\} \times [0,1]$ and $\{1\} \times [0,1]$, respectively. We impose the following conditions:
\begin{itemize}
	\item Coupons are not allowed to intersect the boundary $\partial M$ of $M$. 
	\item The in-boundaries and out-boundaries of ribbons are only allowed to end transversally at the closed or open boundary of $M$ or the in-boundary or out-boundary of a coupon.
	\item These are the only intersections between the boundary of $M$, the ribbons
	and the coupons. 
	\item There is at least one ribbon attached to every coupon.     	
\end{itemize} 
Let $\Ca$ be a finite
ribbon category as defined on page~\pageref{deffiniteribbon} of the introduction.
An \emph{$\Ca$-labeled ribbon graph in $M$} is a ribbon graph $\Gamma$ in $M$ that is non-empty in every connected component of $M$ together with an assignment of a projective object of $\Ca$ to every ribbon. The category $\ARGraphs(M)$ of $\Ca$-labeled ribbon graphs in $M$ is defined as follows:  
\begin{itemize}
	\item Objects are  $\Ca$-labeled ribbon graphs in $M$. 
	\item Morphisms are generated by replacements of ribbon graphs in properly embedded cubes $[0,1]^3$ 
	 (see Figure~\ref{Fig:Ribbon2}): 
		The intersection of the boundary of the cube
	with the ribbon graphs must consists of a non-empty collection of ribbons intersecting $[0,1] \times \{1/2\} \times \{0,1\}  $ transversally.
	We then replace the ribbon graph inside the image of this cube with the ribbon graph that has a coupon placed somewhere on $(0,1) \times \{1/2\} \times (0,1)$ (we will have to make some standard choice for definiteness, but it does not matter which one) and connected to the prescribed boundary intervals by ribbons in $[0,1] \times \{1/2\} \times [0,1]$. 
	We impose the following relations: 
	\begin{itemize}
		\item Any replacement within a single cube having the same
		domain and codomain
		is the identity. 
		\item Replacements within disjoint cubes commute, which allows us to unambiguously define
		the replacement within a disjoint union of finitely many cubes that we call \emph{multicubes}. 
		\item First making the replacement for
		 a multicube sitting in a larger multicube and then making the replacement for the
		 larger multicube is the same as making the evaluation for the larger multicube directly.   
	\end{itemize} 
\end{itemize}

	\begin{figure}[h]	
	\begin{center}
		\begin{overpic}[scale=1]
			{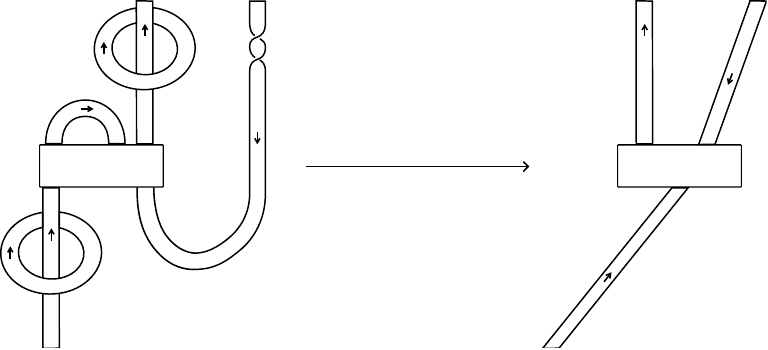}
			\put(49,25){replace by}
		\end{overpic}
	\end{center}
	\caption{A local illustration of a morphism in $\ARGraphs(M)$. The whole ribbon shown here lies in a cube that we do not depict for simplicity.}
	\label{Fig:Ribbon2}
\end{figure}

For a three-dimensional manifold $M$, a \emph{boundary label} consists of a collection of embedded intervals $[0,1]$ in the boundary,
 each of which is labeled by a projective object of $\Ca$ and a sign, encoding whether ribbons are incoming or outgoing at the interval. Every $\Ca$-labeled ribbon graph in $M$ gives rise to a boundary label by restriction to the boundary $\partial M$. For a given boundary label $B$, we denote by $\ARGraphs(M;B)$ the full subcategory of $\ARGraphs(M)$ spanned by those $\Ca$-labeled ribbon graphs restricting to $B$. 

The Reshetikhin-Turaev graphical calculus~\cite[Section~I.2]{turaev} uniquely evaluates an $\Ca$-labeled ribbon diagram in the cube $[0,1]^3$ (where in addition every coupon is labeled with a morphism in $\Ca$ from the incoming labels to the outgoing labels of the coupon) to morphisms in $\Ca$.  Based on the graphical calculus,
one can define the functor 
\begin{align}
	\mathbb{E}_{\Ca}^{M,B}\colon \ARGraphs(M;B) \to \Vect
\end{align}  
sending an $\Ca$-labeled ribbon graph $\Gamma$ to the following tensor product of hom spaces
\begin{align}
	\mathbb{E}_{\Ca}^{M,B} (\Gamma) \coloneqq \bigotimes_{c \ \text{coupon of}\ \Gamma} \Ca \left(X_{c,1}^{\varepsilon^{c}_1}\otimes \dots \otimes X_{c,i_c}^{\varepsilon^{c}_{i_c}}, Y_{c,1}^{\varepsilon^{c}_1}\otimes \dots \otimes Y_{c,j_c}^{\varepsilon^{c}_{j_c}} \right) \ , 
\end{align} 
where 
the $X_{c, \cdot}$ and $Y_{c,\cdot }$ are $i_c$ and $j_c$ many labels for the ribbons ending on the in-boundary or out-boundary of the coupon $c$, respectively; $\varepsilon^{c}_\cdot \in \{\pm \}$ is a sign encoding whether the orientation of ribbon matches the orientation of the coupon 
($+$ for match, $-$ for mismatch), and $X^+\coloneqq X$, and $X^{-}\coloneqq X^{\vee}$. The value on generating morphisms applies the  Reshetikhin-Turaev functor to the ribbon diagram in the corresponding cube.

\begin{definition}[$\text{recovering the definition from \cite{asm}}$]\label{defadmissibleskein}
	Let $\Ca$ be a finite ribbon category and $M$ a three-dimensional oriented manifold with boundary. For every boundary label $B$, the \emph{admissible skein module} is defined as the colimit
	\begin{align}
		\skA (M;B) \coloneqq \colimsub{\Gamma \in \ARGraphs(M;B)}{\mathbb{E}_{\Ca}^{M,B}(\Gamma)} \ \ . \label{eqncolimit}
	\end{align}
\end{definition}  
\begin{remark}[Isotopy invariance]\label{remisotopy}
	The image of a vector in $\mathbb{E}_{\Ca}^{M,B}(\Gamma)$ for some ribbon graph
	$\Gamma \in \ARGraphs(M;B)$ under the structure map $\mathbb{E}_{\Ca}^{M,B}(\Gamma) \to \skA(M;B)$ will be called a \emph{skein in $M$}.
	The skein remains the same if we deform the ribbon describing $\Gamma$ through an isotopy.  
	This is called \emph{isotopy invariance} and is \emph{not} built into the definition because this would create trouble in other places. 
	Inside cubes, the isotopy invariance holds by construction. The general isotopy invariance can be deduced using~\cite[Corollary~1.3]{edwardskirby} as in~\cite[Section~3.1]{fsy-sn}. 
\end{remark}

\begin{remark}
	The 
	 skein modules defined via this colimit coincide with the admissible skein modules from~\cite{asm}, or rather
	  the generalization of them with boundary objects as in \cite{brownhaioun}.
This follows from the same arguments as in the two-dimensional case~\cite[Theorem 3.7]{fsy-sn}, see also \cite[Remark 5.3]{rst}.
	Note that 
	 we just allow projective objects throughout while in \cite{asm} a slightly more general situation allowing non-projective labels is considered
	 as long as there are ribbons labeled with projective objects around. 
	Both versions are equivalent by \cite[Remark 2.4]{brownhaioun}.
\end{remark}

\begin{remark}\label{remasmprop}
	We record some properties of $\skA (M;B)$ proven in~\cite{asm}:
	\begin{pnum}
		\item Suppose that $B$ is a three-dimensional ball with projective boundary labels $P_1,\dots,P_n$ with $n\ge 1$ on its boundary sphere, all of them outgoing for simplicity. Then
		\begin{align}\skA (B;P_1,\dots,P_n)\cong \cat{A}(I,P_1\otimes \dots \otimes P_n)\label{eqnskeindisk}\end{align} by a non-canonical isomorphism
		(this is because on the right hand side an order is chosen while $\skA (B;P_1,\dots,P_n)$ is defined independently of such an order).
		\label{remasmpropi}
		In the colimit description~\eqref{eqncolimit}, this can be seen easily because the right hand side of~\eqref{eqnskeindisk} is simply the value of $\mathbb{E}_{\Ca}^{M,B}$ on a terminal object in $\ARGraphs(M;B)$.
		The case without labels (meaning that the boundary sphere is a free boundary) is more complicated:
		By \cite[Corollary~3.2]{asm} $\skA(B)$ for a three-dimensional ball $B$ is given by the linear dual of the space of possibly degenerate modified traces.
		\item The obvious
		geometric action on $\skA(M;B)$ by the group $\Diff(M)$ of diffeomorphisms of $M$ preserving the orientation and the boundary parametrization descends to an action of the mapping class group $\Map(M):=\pi_0(\Diff(M))$.
		This follows from the isotopy invariance discussed in Remark~\ref{remisotopy}.   \label{remasmpropii}
		\item For a finite ribbon category, the admissible skein module is finite-dimensional.\label{remasmpropiii} 
		\item The skein module functor is monoidal in the sense that it carries
		disjoint unions of manifolds and their labels to tensor products of vector spaces. \label{remasmpropiv}
	\end{pnum}
\end{remark}

\begin{definition}[$\text{following \cite[Section~2.1]{brownhaioun}}$]
	Let $\Ca$ be a finite ribbon category. 
	For a surface $\Sigma$, possibly with boundary,
	we define the $k$-linear \emph{admissible skein category $\skcatA(\Sigma)$} as follows: 
	\begin{itemize}
		\item Objects are boundary labels $B$ in the interior $\Sigma$ such that every connected component of $\Sigma$ contains at least one boundary interval. 
		\item For $B,C \in \skcatA(\Sigma)$, the morphism space is 
		\begin{align}
			\skcatA(\Sigma)(B;C) \coloneqq \skA (\Sigma\times [0,1], B^\vee , C ) \ \ , 
		\end{align}  
		where $B^\vee$ is the boundary label obtained
		 by replacing every element of $B$ by its dual. 
	\end{itemize}  
	Composition is defined by gluing cylinders. 
\end{definition}

\begin{remark}\label{remskeincat}
	The following observations are standard and known in more classical contexts. In the non-semisimple case, the details are in~\cite{brownhaioun}: 
	\begin{pnum}\item	For a two-dimensional disk $D$, the skein category $\skcatA(D)$ is canonically equivalent to $\Proj \cat{A}$. \label{remskeincati}
		\item Let $\partial^{\operatorname{oc}} M$ be the union of the open and closed boundary of $M$. The admissible skein modules introduced in Definition~\ref{defadmissibleskein} assemble into a linear functor 
		\begin{align}
			\skA^\Phi(M;-)\colon \skcatA(\partial^{\operatorname{oc}}M) \to \vect \ \ ,  \label{eqnskeinphimap}
		\end{align}
		where the action on morphisms is by gluing in cylinders.
		\item There is a canonical equivalence $\skcatA(\bar{\Sigma}) \cong \skcatA(\Sigma)^{\opp}$, where $\bar \Sigma$ is $\Sigma$ is opposite orientation.
	\end{pnum}
\end{remark}

The final result that we will discuss in this section is a locality property for skein modules
 called \emph{excision}.

\begin{theorem}[Excision, $\text{based on~\cite[Theorem 4.4.2]{Walker}}$]\label{Thmexcision}
	Let $\cat{A}$ be a finite ribbon category, $M$ an oriented three-dimensional manifold and $M'$ the manifold constructed by gluing $M$ along two oppositely oriented copies of a boundary surface $\Sigma$. Gluing of ribbon graphs  for boundary labels matching up to duality
	induces an isomorphism 
	\begin{align}
		\int^{P\in \skcatA(\Sigma)}\skA (M; X, P ,P^\vee) \ra{\cong} \skA (M'; X) \ \ \text{for} \ \ X\in \skcatA(\partial^{\operatorname{oc}} M' ) \ \ .    \label{eqngluingmap}
	\end{align}    
\end{theorem}

The main ideas of this result can be found in Walker's notes~\cite{Walker},
where this result is not quite stated in the generality needed here. 
Various versions of Walker's statement appear in the literature, 
e.g.\ in~\cite{skeinfin,fsy-sn,sn,brownhaioun}, 
often for skeins in dimension two (also called \emph{string-nets}) and three, sometimes with 
additional details fleshed out that are not contained in the original notes. 
A very detailed treatment taking all the topological intricacies 
seriously and covering several dimensions at once has recently appeared in~\cite{rst}. 
Given that many of the other available proofs for dimension three lack important details and given that the full proof of the more general statement in~\cite{rst}
is the main result of an article of more than 60 pages,
we will include here a --- to the best of our knowledge --- new and relatively short proof of Theorem~\ref{Thmexcision}, in the hope that it will be useful.
The skein theory expert, to whom many of the arguments will be clear, may of course skip the proof. 

\begin{proof}[\slshape Proof of Theorem~\ref{Thmexcision}] 
	\begin{pnum}\item Whether to $\skA (M; X, P ,Q^\vee)$ we add an element in $\skcatA(\Sigma) (Q,P)$ at the first copy of $\Sigma$ at its boundary label $P$ or the second copy with boundary label $Q^\vee$ does not matter after gluing. Both produce the same element  in $\skA (M'; X)$. This means that the map $\skA (M; X, P ,P^\vee) \to \skA (M'; X)$ is dinatural and hence descends to the coend to produce the map~\eqref{eqngluingmap}. \label{excisionstep1}
		\item
		We now consider the case where $M = B_1 \sqcup B_2$ is given by the disjoint union of two three-dimensional balls that are being glued along a two-dimensional disk $D$. Then $M'=B_1 \cup _D B_2$ is a three-dimensional ball again.
		The boundary label $X$ splits up into a boundary label for $B_1$ formed by projective objects $X_1',\dots,X_m'$ whose ribbons end on $\partial B_1$ and a boundary label for $B_2$ formed by projective objects $X_1'',\dots,X_n''$ whose ribbons end on $\partial B_2$. We allow $m$ or $n$ to be zero.	
		We will use the identification 
		$\skA(D)\simeq \Proj\cat{A}$ (Remark~\ref{remskeincat}~\ref{remskeincati}).
		Two cases have to be distinguished: 
		\begin{itemize}
			\item For $m+n\ge 1$, we
			find with Remark~\ref{remasmprop}~\ref{remasmpropi} and~\ref{remasmpropiv} that
			the map~\eqref{eqngluingmap} is  the map
			\begin{align}
				\int^{P\in \Proj \cat{A}}\cat{A}(I,X_1' \otimes \dots \otimes X_m' \otimes P) \otimes \cat{A}(I,X''_1 \otimes\dots \otimes X_n'' \otimes  P^\vee) \to \cat{A}(I,X^{\otimes}) \\ \text{with } X^{\otimes} = X_1'\otimes \dots \otimes X_m' \otimes X_1'' \otimes \dots \otimes X_n''
			\end{align} 
			using the duality and composing over $P$.
			Since $m+n \ge 1$, the statement now follows from the Yoneda Lemma applied to $\Proj \cat{A}$. 
			
			\item 	The case $m=n=0$ requires more work:
			By \cite[Corollary~3.2]{asm} $\skA(B)$ for a three-dimensional ball is given by the linear dual of the space of possibly degenerate modified traces that we may identify with the space  $\cat{A}(\alpha,I)^*$ thanks to \cite[Theorem~6.4]{shibatashimizu}. With the twisted traces in~\cite{shibatashimizu}, see also~\cite{tracesw}, we find
			\begin{equation}
				\int^{P\in \Proj \cat{A}}\cat{A}(I, P) \otimes \cat{A}(I,  P^\vee)\cong \int^{P\in \Proj \cat{A}}\cat{A}( \nakal P, I )^* \otimes \cat{A}(P,I)
			\end{equation} with the left Nakayama functor $\nakal = \int_{X \in \cat{A}} \cat{A}(X,-) \otimes X$. The functor $F= \cat{A}(-,I)^*$ is right exact and hence
			\begin{align} \int^{P\in \Proj \cat{A}}\cat{A}( \nakal P, I )^* \otimes \cat{A}(P,I) &\cong F \int^{P\in \Proj \cat{A}} \nakal P \otimes \cat{A}(P,I)\\ & \cong F \nakal \int^{P \in \Proj \cat{A}} P \otimes \cat{A}(P,I)\\ & \quad \quad \text{($\nakal$ is exact \cite[Corollary 4.7]{fss})}\\& \cong 
				F \nakal \int^{X \in  \cat{A}} X \otimes \cat{A}(X,I) \quad \text{\cite[Proposition~5.1.7]{kl}}\\ &\cong F \nakal I \ .
				\, \end{align}  
			By \cite[Lemma 4.10]{fss} $\nakal I \cong \alpha$ which leaves us with the isomorphism
			\begin{align}
				\int^{P\in \Proj \cat{A}}\cat{A}(I, P) \otimes \cat{A}(I,  P^\vee)\cong F \alpha \cong \cat{A}(\alpha, I)^* \cong \skA(B) \ . 
			\end{align}
			After carefully chasing through all these identifications, we see that this isomorphism is in fact the gluing map, thereby proving excision also for a three-dimensional ball \emph{without} boundary labels.\end{itemize}
		\label{excisionstep2}

		\item 
		In the general case, it is easy to see that~\eqref{eqngluingmap} is surjective because we can take any skein in $M'$ and isotope one of its ribbons such that it intersects $\Sigma$ (seen as surface in $M'$) transversally.
		The isotopy does not change the element in $\skA (M'; X)$ by isotopy invariance, see Remark~\ref{remasmprop}~\ref{remasmpropii}.
		After performing this isotopy, it is clear that the element is in the image of~\eqref{eqngluingmap}. 
		\item The difficult part is  the injectivity of~\eqref{eqngluingmap}. 
		We need to prove that relations holding in $\skA (M'; X)$ that come from the evaluation of the graphical calculus on any cube in $ M'$ 
		already hold in $\int^{P\in \skcatA(\Sigma)}\skA (M; X, P ,P^\vee)$. 
		For this, it suffices to prove that the relations for skeins in any three-dimensional closed ball $B$
		in $M'$ already hold in $\int^{P\in \skcatA(\Sigma)}\skA (M; X, P ,P^\vee)$
		because every properly embedded cube can be slightly 
		enlarged to an embedding of a three-dimensional ball $B$. 
		(The evaluation on a cube induces only relations on skeins that are transversal to the boundary of the cube. We can therefore arrange the skeins to be transversal to the boundary of the ball as well.)
		
		This is clear if $B$ is the
		image of a ball in $M$. If this is not the case,
		then, in analogy to  the proof of~\cite[Theorem~3.2]{sn}, it is enough to prove that $B$ can be obtained by gluing balls that are each image of a ball in $M$ along two-dimensional disks (because then the situation can be played back to the one of step~\ref{excisionstep2}). We now prove that $B$ can indeed be cut this way:
		We first adjust the shape of $B$ slightly such that the intersection of $\partial B$ with $\Sigma$ is transversal. This is possible because the evaluation on $B$ only gives relations for skeins ending transversally on $\partial B$, so the shape of $B$ can be adjusted slightly without changing the relations induced by evaluation on $B$. Now $S=B\cap \Sigma$ is a two-dimensional manifold thanks to transversality, and it cuts $B$ into two three-dimensional manifolds $K$ and $K'$, i.e.\ $B=K \cup_S K'$. In contrast to the two-dimensional situation, $K$ and $K'$ do not need to be a disjoint union of balls (in this case, we would be done), which makes the treatment of the three-dimensional situation somewhat different.
		
		As a remedy, we choose a two-dimensional triangulation $t$ of $S$ and three-dimensional triangulations $T$ for $K$ and $T'$ for $K'$ extending the one on $S$. This is possible by~\cite[Variant~7]{lurie-t}. 
		Now $T\cup T'$ is a three-dimensional triangulation of $B$, with each of the triangles contained in $M$. This finishes the proof. 
	\end{pnum}
\end{proof}

\begin{remark}
	In this remark, we want to give a brief topological argument for the relation between modified traces and (co)integrals for Hopf algebras as presented in~\cite{bbg,shibatashimizuintegral,bgr}:
	If we cut the three-dimensional sphere into two-hemispheres, we find with excision (Theorem~\ref{Thmexcision}) and \cite[Corollary~3.2]{asm}
	\begin{align}
		\int^{P \in \Proj \cat{A}} \cat{A}(I,P) \otimes \cat{A}(P,I)    \cong \{   \text{space of two-sided modified traces}\}^* \ . 
	\end{align} 
	The coend on the left side can be restricted to a projective generator without changing its values, see e.g.\ \cite[Theorem~2.9]{dva}, which means that,
	if $\cat{A}$ is given by finite-dimensional modules over a finite-dimensional ribbon Hopf algebra $H$, we obtain 
	\begin{align} \Hom_H(k,H) \otimes_H k   \cong \{   \text{space of two-sided modified traces}\}^* \ . 
	\end{align} Note that $\Hom_H(k,H)$ is the space $\mathfrak{I}_\ell(H)$ of left integrals for $H$ (depending on the conventions, they might be referred to as cointegrals). The tensor product $\mathfrak{I}_\ell(H)\otimes_H k$ is zero if $H$ is not unimodular. If $H$ is unimodular, it is the space of two-sided integrals.
\end{remark}

\section{The comparison result}
The language and the tool to formulate and prove the comparison result is the notion of an \emph{ansular functor} introduced in~\cite{mwansular} with preparations in~\cite{cyclic}. 
An ansular functor is a consistent systems of handlebody group representations. It is formally defined as a modular algebra over the modular operad $\Hbdy$ of handlebodies. It takes values in a symmetric monoidal bicategory that for our applications will consist of linear categories.
The most reasonable choice in the context of this article is the symmetric monoidal bicategory $\Rexf$ of finite categories over our fixed algebraically closed field $k$, right exact functors and linear natural transformations.
The monoidal product is the Deligne product $\boxtimes$. 
When partially spelling out the definition, an ansular functor $\mathfrak{B}$ consists of the following: \begin{itemize}
	\item It has an underlying category $\cat{A}\in\Rexf$ with a non-degenerate symmetric pairing $\kappa : \cat{A}\boxtimes\cat{A} \to \vect$. Non-degeneracy means that we have a map $\Delta : \vect \to \cat{A}\boxtimes \cat{A}$, a \emph{coevaluation}, satisfying together with $\kappa$ the zigzag relations for duality up to isomorphism.
	The symmetry of $\kappa$ means that $\kappa$ is a homotopy $\mathbb{Z}_2$-fixed point with respect to the action on $\kappa$ through the symmetric braiding of $\boxtimes$. (This structure is not specific to ansular functors; it is just needed to make the endomorphism operad of $\cat{A}$ cyclic.)
	\item For any handlebody $H$ with $n \ge 0$ embedded disks in its boundary surface, we obtain a right exact functor
	\begin{align}
		\mathfrak{B}(H): \cat{A}^{\boxtimes n} \to \vect 
	\end{align}  carrying a representation of the handlebody group $\Map(H)$ (the group of connected components of the topological group of orientation preserving diffeomorphisms which fix the embedded disks point-wise) subject to the excision property that the functor
	$\mathfrak{B}(H'):\cat{A}^{\boxtimes (n-2)} \to \vect$  
	for a handlebody $H'$ obtained from $H$ by sewing a pair of disks together (assuming $n\ge 2$) is obtained by inserting the coevaluation object $\Delta = \Delta' \boxtimes \Delta'' \in \cat{A}\boxtimes\cat{A}$ 
	(here $\Delta = \Delta' \boxtimes \Delta''$ is Sweedler notation; it should not suggest that $\Delta$ is a pure tensor)
	into the slots associated to these disks.
	This means
	 \begin{align}\label{eqngluing} \mathfrak{B}(H')\cong \mathfrak{B}(H; \dots,\Delta',\dots,\Delta'',\dots)\end{align} by a specific isomorphism that is also part of the data.
\end{itemize}

We will now make the first observation: For a finite ribbon category $\cat{A}$, the admissible skein modules produce an ansular functor.
To this end, let $H$ be a handlebody with $n\ge 0$ embedded disks.
Now
\begin{align} \skA(H): (\Proj \cat{A})^{\otimes n} \to \vect \ , \quad (P_1,\dots,P_n) \mapsto \skA(M;P_1,\dots,P_n) \label{eqnskHproj} \end{align} is a linear functor with $\Map(H)$-action. 
In order to obtain something that could possibly match the above description, we would need a right exact functor out of $\cat{A}^{\boxtimes n}$,
but this can easily be done by applying the finite free cocompletion, a standard technique that, in the finite setting, is recalled in detail e.g.\ in~\cite[Section~5]{sn}: We can recover $\cat{A}$ from $\Proj \cat{A}$ be adding finite colimits, i.e.\ by passing from $\Proj \cat{A}$ to the category of functors $(\Proj \cat{A})^\opp \to \vect$; in other words, $\cat{A}$ is the finite free cocompletion of $\Proj\cat{A}$. 
The linear functor~\eqref{eqnskHproj} then extends uniquely to a	 right exact functor
\begin{align} \SkA(H): \cat{A}^{\boxtimes n} \to \vect \ .  \label{eqnskHext} \end{align}
By setting $H=\mathbb{D}^2 \times [0,1]$, we obtain a non-degenerate symmetry pairing. If needed, we may always use the pairing to move $\cat{A}^{\boxtimes n}$ from the left to the right  in~\eqref{eqnskHext}.
This data gives us indeed an ansular functor:

\begin{proposition}\label{propasmansular}
	For a finite ribbon category $\cat{A}$, the freely cocompleted admissible skein modules 
	\begin{align} \SkA(H;-):\cat{A}^{\boxtimes n} \to \vect \end{align}
	for handlebodies with $n\ge 0$ embedded disks form a $\Rexf$-valued ansular functor. 
\end{proposition}

\begin{proof}
	This is mostly a straightforward verification, with the only non-obvious point being the excision property~\eqref{eqngluing}
	that we need to play back to Theorem~\ref{Thmexcision}.
	This can be done in a rather indirect way: Let us recall from~\cite[Definition~5.4]{sn}
	the symmetric monoidal bicategory $\catf{Bimod}^\catf{f}$:\begin{itemize}\item  The objects are  $k$-linear categories $\cat{A}$ whose finite free cocompletion is a finite category.
		\item A 1-morphism $\cat{A} \to \cat{B}$ is a linear functor $M: \cat{A}^\opp \otimes \cat{B} \to\vect$ (also called an \emph{$\cat{A}$-$\cat{B}$-bimodule}).
		\item 2-morphisms are linear natural transformations of bimodules.
	\end{itemize}
	There is a well-known canonical equivalence of symmetric monoidal bicategories $\catf{Bimod}^\catf{f}\simeq \Rexf$ (see e.g.\ \cite[Proposition~5.6]{sn})
	under which $\skA(H;-)$ in~\eqref{eqnskHproj} corresponds to $\SkA(H;-)$ in~\eqref{eqnskHext}.
	In $\catf{Bimod}^\catf{f}$, the excision property~\eqref{eqngluing}
	amounts to 
	\begin{align}
		\skA(H';-)\cong \int^{P\in\Proj \cat{A}} \skA(H;\dots,P^\vee,\dots,P,\dots) \ , 
	\end{align}
	which is now a special case of Theorem~\ref{Thmexcision}. This finishes the proof.
\end{proof}
\begin{remark}
	Any handlebody without embedded disks in its boundary can be constructed by gluing a three-dimensional 
	ball with one embedded disk in its boundary to a handlebody $H$ with one boundary disk. The three-dimensional
	ball with one embedded boundary disk corresponds to the monoidal unit $I\in\cat{A}$ and hence the gluing property for ansular functors implies that $\SkA (H;I)$ is the admissible skein module for the handlebody without embedded disks.   
\end{remark}

If we restrict the ansular functor from Proposition~\ref{propasmansular} to genus zero,
we obtain the following:
\begin{proposition}\label{propskagenuszero}
	For a finite ribbon category $\cat{A}$, the ribbon Grothendieck-Verdier category obtained by restriction of the ansular functor from 
	Proposition~\ref{propasmansular} built from admissible skein modules is the balanced braided category $\cat{A}$ with $\alpha^{-1}$ as ribbon dualizing object whose Grothendieck-Verdier duality $D=\alpha^{-1}\otimes -^\vee$ is the shift of the rigid duality by $\alpha^{-1}$. 
\end{proposition}

We denote this ribbon Grothendieck-Verdier category by $\cat{A}_{\alpha^{-1}}$.

\begin{proof}Let us calculate the genus zero restriction of the ansular functor from Proposition~\ref{propasmansular}, which by \cite[Theorem~5.9]{mwansular} is a ribbon Grothendieck-Verdier category:
	For a three-dimensional ball with $n\ge 1$ disks embedded in its boundary labeled by projective objects $P_1,\dots,P_n$, we find
	\begin{align}
		\skA(B;P_1,\dots,P_n)&\cong \cat{A}(I,P_1\otimes \dots \otimes P_n) \qquad \text{(Remark~\ref{remasmprop}~\ref{remasmpropi})} \\
		&\cong 	\cat{A}(I,\nakar \nakal( P_1\otimes \dots \otimes P_n)) \qquad \text{\cite[Corollary 4.7]{fss}} \\
		& \qquad\qquad  \text{with the left / right Nakayama functor~\cite{fss}} \\
		& \qquad \qquad \nakal = \int_{X \in \cat{A}} \cat{A}(X,-) \otimes X\ , \quad \nakar = \int^{X \in \cat{A}} \cat{A}(-,X)^* \otimes X \\
		&\cong \cat{A}(\nakal ( P_1\otimes \dots \otimes P_n),I )^* \qquad \text{(\cite[Section~3]{shibatashimizu} \& \cite[Section~2]{tracesw})}\\
		&\cong \cat{A}(  P_1\otimes \dots \otimes P_n,\nakar I )^*
		\qquad \text{(\cite[Lemma 3.16]{fss})}
		\\ &\cong 
		\cat{A}(  P_1\otimes \dots \otimes P_n,\alpha^{-1} )^*
		\qquad \text{(\cite[Lemma 4.10 \& 4.11]{fss})}
		\ . 
	\end{align}
	Since $\SkA(B;-)$ is the unique right exact extension of this functor, we find
	\begin{align}
		\SkA(B;X_1,\dots,X_n)\cong \cat{A}(  X_1\otimes \dots \otimes X_n,\alpha^{-1} )^*  \label{eqnSkAgenusnull}
	\end{align}
for $X_1,\dots,X_n\in\cat{A}$.
	With the description of spaces of conformal blocks for ansular functors, which is given in \cite[Theorem~7.8]{cyclic} and \cite[Corollary~6.3]{mwansular} in the $\Lexf$-valued case and adapted in~\cite[Corollary~8.1]{brochierwoike} to the $\Rexf$-valued case, we deduce that the underlying non-cyclic framed $E_2$-algebra of the genus zero restriction of $\SkA$ is the balanced braided category $\cat{A}$.
	It remains to identify the cyclic structure. In~\cite[Theorem~4.2]{mwcenter} all possible ribbon
	Grothendieck-Verdier dualities on a balanced braided category $\cat{A}$ are classified: If $\cat{A}$ is rigid (as in the case at hand), then the rigid duality is a ribbon Grothendieck-Verdier duality, and all other possible
	ribbon Grothendieck-Verdier dualities are given by $D=\beta \otimes -^{\vee}$, where $\beta$ is an invertible object in the balanced Müger center 
	\begin{align}
		Z_2^\catf{bal}(\cat{A}) := \left\{ X \in \cat{A} \  | \ 
		c_{Y,X}c_{X,Y}=\id_{X\otimes Y}\ \forall \ Y \in \cat{A} , \ 
		\theta_X=\id_X \right\}
	\end{align}
	of $\cat{A}$. In other words, the possible ribbon Grothendieck-Verdier dualities form a torsor over the Picard group of $Z_2^\catf{bal}(\cat{A})$.
	We conclude that the cyclic framed $E_2$-structure on the genus zero restriction of $\SkA$ is determined by knowing the underlying balanced braided category (which is $\cat{A}$) and the invertible object in $Z_2^\catf{bal}(\cat{A})$ that the rigid duality needs to be twisted by. This object is $\alpha^{-1}$ as we see easily with~\eqref{eqnSkAgenusnull} and \cite[Corollary~8.1]{brochierwoike}. 
\end{proof}

\begin{remark}
	Logically, we do not need to prove that $\alpha$ or equivalently $\alpha^{-1}$ actually lie in $Z_2^\catf{bal}(\cat{A})$.
	If $\SkA$ is an ansular functor, then this \emph{must}
	 be true by \cite[Theorem~4.2]{mwcenter}.
	As a sanity check, let us also briefly see in a purely algebraic way why this is true:
	By~\cite[Corollary~8.10.8]{egno} the distinguished invertible object $\alpha$ trivially double braids with all objects. 
	Moreover, by \cite[Lemma 4.10 \& 4.11]{fss}
	$\alpha \cong \nakal I=\int_{X \in \cat{A}} \cat{A}(X,I)\otimes X$, which by naturality of $\theta$ and $\theta_I=\id_I$ implies $\theta_\alpha = \id_\alpha$.
\end{remark}

We can now state and prove our main result:
\begin{theorem}[Comparison result]\label{thmahatasm}
	Let $\cat{A}$ be a finite ribbon category and $H$ a handlebody with $n\ge 0$ disks embedded in its boundary surface. Then there is a canonical natural $\Map(H)$-equivariant isomorphism of functors
	\begin{align}
		\widehat{\cat{A}}_{\alpha^{-1}}(H;-) \ra{\cong} \SkA(H;-) : \cat{A}^{\boxtimes n} \to \vect
	\end{align}
	between the value of the modular extension $\widehat{\cat{A}}_{\alpha^{-1}}$ of the balanced braided category 
	$\cat{A}$ with Grothendieck-Verdier duality $D=\alpha^{-1}\otimes -^\vee$  on $H$ and the cocompleted
	admissible skein module $\SkA(H;-)$. 
	In the case $n=0$ (no embedded disks), the completion step for admissible skein modules is not needed, and we obtain a $\Map(H)$-equivariant isomorphism
	\begin{align}
		\widehat{\cat{A}}_{\alpha^{-1}}(H) \ra{\cong} \skA(H) 
	\end{align}
	of vector spaces
	for the uncompleted admissible skein modules $\skA(H) $.
\end{theorem}

The above result follows directly from the next result that uses slightly more technical language that for the statement of Theorem~\ref{thmahatasm} we want to avoid:

\begin{theorem}\label{Thmcomparisionansular}
	Let $\cat{A}$ be a finite ribbon category.
	Then the following constructions agree, up to canonical equivalence,
	as 
	$\Rexf$-valued ansular functors:
	\begin{pnum}
		\item The modular extension $\widehat{\cat{A}}_{\alpha^{-1}}$ of $\cat{A}_{\alpha^{-1}}$ seen as $\Rexf$-valued modular $\Hbdy$-algebra, i.e.\ the ansular functor associated to $\cat{A}_{\alpha^{-1}}$.
		\item The admissible skein modules for $\cat{A}$, seen as 
		$\Rexf$-valued ansular functor after a finite free cocompletion. 
	\end{pnum}
\end{theorem}

\begin{proof}
	We need to compare $\widehat{\cat{A}}_{\alpha^{-1}}$ with the ansular functor $\SkA$
	built in Proposition~\ref{propasmansular} from admissible skein modules and finite free cocompletion.
	By \cite[Theorem 5.7 \& 5.9]{mwansular} two ansular functors are equivalent if and only if their genus zero restrictions are equivalent as cyclic framed $E_2$-algebras, i.e.\ as ribbon Grothendieck-Verdier categories. 
	As the genus zero restriction of $\widehat{\cat{A}}_{\alpha^{-1}}$, we obtain $\cat{A}_{\alpha^{-1}}$ by construction.
	For $\SkA$, the result is the same by Proposition~\ref{propskagenuszero}.
	This finishes the proof.
\end{proof}

\begin{corollary}
	Let $\Ca$ be a finite ribbon category and $H_{g,n}$ a genus $g$ 
	handlebody with $n\ge 0$ embedded boundary disks with 
	boundary labels $X_1,\dots,X_n$. The vector space of admissible skeins with these boundary labels is, after finite free cocompletion, given by
	\begin{align}
		\SkA(H_{g,n};X_1,\dots,X_n) \cong \Ca (X_1\otimes \dots X_n \otimes \mathbb{A}^{\otimes g},\alpha^{-1})^* \ \ ,\label{eqnbackgroundc}
	\end{align} 
	where $\mathbb{A}=\int_{X\in \Ca} X^\vee \otimes X \in \Ca $ is the canonical end of $\Ca$ and $\alpha^{-1}$ the inverse of the distinguished invertible object. 
\end{corollary} 

\begin{proof}
	This  directly follows from Theorem~\ref{thmahatasm} if we compare
	\cite[Theorem~7.8]{cyclic} and \cite[Corollary~6.3]{mwansular}
	(where the formulae are given in the $\Lexf$-valued case; see \cite[Corollary~8.1]{brochierwoike} for the $\Rexf$-valued case).
\end{proof}

\begin{remark}[The distinguished invertible object as a background charge]\label{rembcharge}
	The vector space $\Ca (X_1\otimes \dots X_n \otimes \mathbb{A}^{\otimes g},\alpha^{-1})^*$
	on the right hand side of~\eqref{eqnbackgroundc} is isomorphic
	to $\Ca (\alpha\otimes X_1\otimes \dots X_n \otimes \mathbb{A}^{\otimes g},I)^*$. In other words, it is the `usual' space of conformal blocks that we expect, but with $\alpha$ as an additional boundary label.
	This is, in different contexts, a well-known phenomenon referred to as \emph{background charge} in~\cite[Remark~7.2]{Non-s-string-nets}, see also	\cite[Remark 8.9]{sn}.
	Theorem~\ref{thmahatasm} tells us that for the comparison we need to add the background charge $\alpha^{-1}$ to the ansular functor or, equivalently the background charge $\alpha$ to the admissible skeins.
\end{remark}

We turn now to the comparison of the skein \emph{module structure}. As was already explained in the introduction, this requires the following statement:

\begin{theorem}\label{thmgenskeinmodule}
	For a unimodular finite ribbon category $\cat{A}$ and a handlebody $H$ with $n\ge 0$ embedded disks in its boundary surface $\Sigma$ (with as usual the interior of the disks removed), the generalized handlebody skein module
	$\PhiA(H):\int_\Sigma \cat{A} \to \cat{A}^{\boxtimes n}$ agrees, under the equivalence~\eqref{eqnequivfhskein} between factorization homology and the freely cocompleted skein category,
	with the map $\SkA^\Phi(H;-) : \SkCatA(\Sigma) \to \cat{A}^{\boxtimes n}$ induced by the admissible skein modules
	as a module map over $\int_{\partial \Sigma \times [0,1]} \cat{A}$. 
\end{theorem}

The map $\SkA^\Phi(H;-)$ was already briefly explained in the introduction. Its precise origin will be discussed in the proof.

\begin{proof}
	Let $\varphi : (\mathbb{D}^2)^{\sqcup m} \to \Sigma$ 
	be an embedding of $m\ge 0$ disks into $\Sigma$ and $\varphi _ * : \cat{A}^{\boxtimes m}\to \int_\Sigma \cat{A}$ the associated structure map that factorization homology, by its definition as a colimit, comes equipped with.
	The map $\PhiA(H): \int_\Sigma \cat{A}\to\cat{A}^{\boxtimes n}$ is determined by $\PhiA(H)\varphi_* = \widehat{\cat{A}}(H^\varphi)$, where $H^\varphi$ is the handlebody $H$, but with all $n$ embedded disks seen as outgoing and $m$ incoming disks added as prescribed by the embedding $\varphi$. 
	With unimodularity and Theorem~\ref{thmahatasm}, we find
	$\widehat{\cat{A}}(H^\varphi)\cong \SkA(H^\varphi)$. 
	But the maps  $\SkA(H^\varphi)$ are exactly the ones inducing $\SkA^\Phi(H):\SkCatA(\Sigma)\to\cat{A}^{\boxtimes n}$ under the equivalence $\SkCatA(\Sigma)\simeq \int_\Sigma \cat{A}$, see \cite[Section~2.3 \& Theorem~3.10]{brownhaioun}, see also~\eqref{eqnskeinphimap} (whether $\cat{A}^{\boxtimes n}$ appears here in the source or the target is just a matter of taste because of the cyclic structure, see the comments after~\eqref{eqnskHext}).
	The statement of the maps agreeing as module maps follows if we extend the comparison just made to the construction in \cite[Proposition 4.3]{brochierwoike}.
\end{proof}

Since $\PhiA(H)\cat{O}_\Sigma^\cat{A}\cong \widehat{\cat{A}}(H)$~\cite[Theorem~4.2]{brochierwoike}, $\widehat{\cat{A}}(H)$ carries an action of
$\SkAlg_\cat{A}(\Sigma)=\End_{\int_\Sigma \cat{A}}(\cat{O}_\Sigma^\cat{A})$.
On the other hand, on the skein module $\skA(H)$ the action of a non-unital skein algebra is described in \cite[Proposition 2.5]{asm}.
By \cite[Section~2.2 \& Remark 2.19]{brownhaioun} this action extends to a $\SkAlg_\cat{A}(\Sigma)$-action on the cocompleted skein module $\SkA(H)$; this also means in particular that, in hindsight, there is actually an action of $\SkAlg_\cat{A}(\Sigma)$ on $\skA(H)=\SkA(H)$ in the case of no embedded disks.
Now Theorem~\ref{thmgenskeinmodule} implies:

\begin{corollary}\label{corskeinmodules}
	For a unimodular finite ribbon category $\cat{A}$, and any handlebody without embedded boundary disks $H$, there is an isomorphism $\widehat{\cat{A}}(H)\cong \SkA(H)$ of skein modules, i.e.\ as modules over  $\SkAlg_\cat{A}(\Sigma)$.
\end{corollary}

For future reference, we record the following statement:

\begin{corollary}[Excision for admissible skein modules in terms of the generalized skein modules $\PhiA$]\label{corexcisionphi}
	Let $\cat{A}$ be a unimodular finite ribbon category, $M$ a closed oriented three-dimensional manifold and 
	$M=H' \cup_\Sigma H$ a Heegaard splitting for $M$, with gluing along a surface $\partial \bar{H'} = \Sigma = \partial H$.
	Then the admissible skein module for $M$
	(both the uncompleted and completed version)
	can be obtained as a coend over  the generalized skein modules $\PhiA(H)$ and $\PhiA(H')$, i.e.\ there is a canonical isomorphism
	\begin{align}
		\int^{P \in \Proj \int_\Sigma \cat{A}} \PhiA(H';P^\vee)\otimes \PhiA(H;P) \ra{\cong}
		\skA(M) \ . 
	\end{align}
\end{corollary}

\begin{proof}
	Theorem~\ref{thmgenskeinmodule} gives us
	\begin{align}
		\int^{P \in \Proj \int_\Sigma \cat{A}} \PhiA(H';P^\vee)\otimes \PhiA(H;P)
		&\cong \int^{P \in \Proj \SkCatA(\Sigma)} \SkA^\Phi(H';P^\vee)\otimes\SkA^\Phi(H;P) \ . 
	\end{align}
	By the definition of the free finite cocompletion \begin{align}\int^{P \in \Proj \SkCatA(\Sigma)} \SkA^\Phi(H';P^\vee)\otimes\SkA^\Phi(H;P)\cong \int^{P \in  \skcatA(\Sigma)} \skA^\Phi(H';P^\vee)\otimes\skA^\Phi(H;P) \ . \end{align}
	Now Theorem~\ref{Thmexcision} implies the assertion.
\end{proof}

\begin{remark}[Two-dimensional conformal field theory and three-dimensional skein theory]
	In~\cite[Remark~9.13]{microcosm} a correspondence between two-dimensional full conformal field theory in genus zero and three-dimensional skein theory for handlebodies is deduced from the modular microcosm principle.
	In that context, the monodromy data of the conformal field theory is given by a ribbon Grothendieck-Verdier category, and the skein theory has to be defined entirely in terms of factorization homology.
	If $\cat{A}$ is a unimodular finite ribbon category, Theorem~\ref{thmahatasm} tells us that there is a skein description in a more classical sense based on a graphical calculus. This will be exploited elsewhere.
\end{remark}

\vspace*{0.3cm} \noindent  \textsc{Perimeter Institute,  N2L 2Y5 Waterloo, Canada} \\ \textsc{Université Bourgogne Europe, CNRS, IMB UMR 5584, F-21000 Dijon, France}

\end{document}

%% file: diagrams.tikzstyles
\tikzstyle{black dot}=[fill=black, draw=black, shape=circle, minimum size=3pt, inner sep=0pt]
\tikzstyle{black dot small}=[fill=black, draw=black, shape=circle, minimum size=3pt, inner sep=0pt]

\tikzstyle{big white circle}=[fill=white, draw=black, shape=circle, minimum width=0.75cm]
\tikzstyle{white dot big}=[fill=white, draw=black, shape=circle, inner sep=1pt]
\tikzstyle{white dot}=[fill=white, draw=black, shape=circle, minimum size=3pt, inner sep=0pt]
\tikzstyle{flat box}=[fill=white, draw=black, shape=rectangle, minimum width=2.5cm, minimum height=0.5cm]
\tikzstyle{square}=[fill=white, draw=black, shape=rectangle]
\tikzstyle{flat box 2}=[fill=white, draw=black, shape=rectangle, minimum height=0.5cm, minimum width=1.0cm]
\tikzstyle{over }=[front]
\tikzstyle{theta}=[fill=black, draw=black, shape=ellipse, minimum height=6pt, minimum width=6pt, inner sep=0pt]
\tikzstyle{thetabig}=[fill=black, draw=black, shape=ellipse, minimum width=1cm, minimum height=0.01cm]
\tikzstyle{thetainv}=[fill=white, draw=black, shape=ellipse, minimum height=6pt, minimum width=6pt, inner sep=0pt]
\tikzstyle{thetabinv}=[fill=white, draw=black, shape=ellipse, minimum width=1cm, minimum height=0.01cm]
\tikzstyle{black over}=[fill=white, draw=black, shape=circle]

\tikzstyle{mid arrow}=[-, postaction={on each segment={mid arrow}}]
\tikzstyle{end arrow}=[->]
\tikzstyle{red mid arrow}=[-, draw={rgb,255: red,214; green,42; black,51}, postaction={on each segment={mid arrow}}, line width=1.1pt]
\tikzstyle{reddots}=[-,dotted, draw={rgb,255: red,214; green,42; blue,51},line width=1pt]
\tikzstyle{blue}=[-, draw=black, line width=1.1pt]
\tikzstyle{blue mid arrow}=[-, draw={rgb,255: red,23; green,37; black,167}, postaction={on each segment={mid arrow}}, line width=1.1pt]
\tikzstyle{over}=[-, link]
\tikzstyle{mapsto}=[{|->}]
\tikzstyle{blue arrow}=[draw=blue, ->, line width=1.1pt]
\tikzstyle{blue}=[-, draw=blue, line width=1.1pt]
\tikzstyle{black over}=[-, link2, line width=1.1pt]

%% file: skein-in-dim-3.bbl
\small	
\begin{thebibliography}{MSWY23}
	
	\bibitem[AF15]{AF}
	D.~Ayala and J.~Francis.
	\newblock {Factorization homology of topological manifolds}.
	\newblock {\em J. Top.}, 8(4):1045--1084, 2015.
	
	\bibitem[And06]{andersenfaithful}
	J.~E. Andersen.
	\newblock {Asymptotic faithfulness of the quantum SU(n) representations of the
		mapping class groups}.
	\newblock {\em Ann. Math.}, 163(1):347--368, 2006.
	
	\bibitem[BBG21]{bbg}
	A.~Beliakova, C.~Blanchet, and A.~Gainutdinov.
	\newblock {Modified trace is a symmetrised integral}.
	\newblock {\em Sel. Math. New Ser.}, 27(31), 2021.
	
	\bibitem[BD13]{bd}
	M.~Boyarchenko and V.~Drinfeld.
	\newblock A duality formalism in the spirit of {G}rothendieck and {V}erdier.
	\newblock {\em Quantum Top.}, 4(4):447--489, 2013.
	
	\bibitem[BGR20]{bgr}
	J.~Berger, A.~Gainutdinov, and I.~Runkel.
	\newblock Modified traces for quasi-hopf algebras.
	\newblock {\em J. Alg.}, 548:96--119, 2020.
	
	\bibitem[BH24]{brownhaioun}
	J.~Brown and B.~Ha\"ioun.
	\newblock Skein categories in non-semisimple settings.
	\newblock arXiv:2406.08956 [math.QA], 2024.
	
	\bibitem[BJSS21]{bjss}
	A.~Brochier, D.~Jordan, P.~Safronov, and N.~Snyder.
	\newblock {Invertible braided tensor categories}.
	\newblock {\em Alg. Geom. Top.}, 21(4):2107--2140, 2021.
	
	\bibitem[BK01]{baki}
	B.~Bakalov and A.~Kirillov.
	\newblock {\em Lectures on tensor categories and modular functors}, volume~21
	of {\em University Lecture Series}.
	\newblock Am. Math. Soc., 2001.
	
	\bibitem[BW22]{brochierwoike}
	A.~Brochier and L.~Woike.
	\newblock A classification of modular functors via factorization homology.
	\newblock Accepted for publication in \emph{Geom. Top.}, arXiv:2212.11259
	[math.QA], 2022.
	
	\bibitem[BZBJ18a]{bzbj}
	D.~Ben-Zvi, A.~Brochier, and D.~Jordan.
	\newblock {Integrating quantum groups over surfaces}.
	\newblock {\em J. Top.}, 11(4):874--917, 2018.
	
	\bibitem[BZBJ18b]{bzbj2}
	D.~Ben-Zvi, A.~Brochier, and D.~Jordan.
	\newblock {Quantum character varieties and braided module categories}.
	\newblock {\em Selecta Math.}, 24(35):4711--4748, 2018.
	
	\bibitem[CGPM23]{asm}
	F.~Costantino, N.~Geer, and B.~Patureau-Mirand.
	\newblock Admissible skein modules.
	\newblock arXiv:2302.04493 [math.GT], 2023.
	
	\bibitem[Coo23]{cooke}
	J.~Cooke.
	\newblock Excision of skein categories and factorisation homology.
	\newblock {\em Adv. Math.}, 414:108848, 2023.
	
	\bibitem[Cos04]{costello}
	K.~Costello.
	\newblock The {A}-infinity operad and the moduli space of curves.
	\newblock arXiv:math/0402015 [math.AG], 2004.
	
	\bibitem[EGNO15]{egno}
	P.~Etingof, S.~Gelaki, D.~Nikshych, and V.~Ostrik.
	\newblock {\em Tensor categories}, volume 205 of {\em Math. Surveys Monogr.}
	\newblock Am. Math. Soc., 2015.
	
	\bibitem[EK71]{edwardskirby}
	R.~Edwards and R.~Kirby.
	\newblock {Deformations of spaces of imbeddings}.
	\newblock {\em Ann. Math.}, 93:63--88, 1971.
	
	\bibitem[ENO04]{eno-d}
	P.~Etingof, D.~Nikshych, and V.~Ostrik.
	\newblock {An analogue of Radford's $S^4$ formula for finite tensor
		categories}.
	\newblock {\em Int. Math. Res. Not.}, 54:2915--2933, 2004.
	
	\bibitem[EO04]{etingofostrik}
	P.~Etingof and V.~Ostrik.
	\newblock {Finite tensor categories}.
	\newblock {\em Mosc. Math. J.}, 4(3):627--654, 2004.
	
	\bibitem[FSS20]{fss}
	J.~Fuchs, G.~Schaumann, and C.~Schweigert.
	\newblock {Eilenberg-Watts calculus for finite categories and a bimodule
		Radford $S^4$ theorem}.
	\newblock {\em Trans. Am. Math. Soc.}, 373:1--40, 2020.
	
	\bibitem[FSY25]{fsy-sn}
	J.~Fuchs, C.~Schweigert, and Y.~Yang.
	\newblock {String-net models for pivotal bicategories}.
	\newblock {\em Theory Appl. Cat.}, 44(17):474--543, 2025.
	
	\bibitem[FWW02]{fww}
	M.~Freedman, K.~Walker, and Z.~Wang.
	\newblock {Quantum SU(2) faithfully detects mapping class groups modulo
		center}.
	\newblock {\em Geom. Top.}, 6:523--539, 2002.
	
	\bibitem[Gia11]{giansiracusa}
	J.~Giansiracusa.
	\newblock {The framed little 2-discs operad and diffeomorphisms of
		handlebodies}.
	\newblock {\em J. Top.}, 4(4):919--941, 2011.
	
	\bibitem[GJS23]{skeinfin}
	S.~Gunningham, D.~Jordan, and P.~Safronov.
	\newblock {The finiteness conjecture for skein modules}.
	\newblock {\em Invent. Math.}, 232:301--363, 2023.
	
	\bibitem[GK95]{gk}
	E.~Getzler and M.~Kapranov.
	\newblock Cyclic operads and cyclic homology.
	\newblock In R.~Bott and S.-T. Yau, editors, {\em Conference proceedings and
		lecture notes in geometry and topology}, pages 167--201. Int. Press, 1995.
	
	\bibitem[HP92]{hp92}
	J.~Hoste and J.~Przytycki.
	\newblock A survey of skein modules of 3-manifolds.
	\newblock In A.~Kawauchi, editor, {\em Proceedings of the International
		Conference on Knot Theory and Related Topics}, pages 363--379. De Gruyter,
	1992.
	
	\bibitem[KL01]{kl}
	T.~Kerler and V.~V. Lyubashenko.
	\newblock {\em Non-Semisimple Topological Quantum Field Theories for
		3-Manifolds with Corners}, volume 1765 of {\em Lecture Notes in Math.}
	\newblock Springer, 2001.
	
	\bibitem[Lur09]{lurie-t}
	J.~Lurie.
	\newblock Existence of triangulations ({L}ecture 4).
	\newblock Notes available at
	https://www.math.ias.edu/\~{}lurie/937notes/937Lecture4.pdf, 2009.
	
	\bibitem[Mas03]{masbaumja}
	G.~Masbaum.
	\newblock Quantum representations of mapping class groups.
	\newblock In: Groupes et Géométrie, Journée annuelle 2003 de la SMF,
	p.~19--36, 2003.
	
	\bibitem[MR95]{masbaumroberts}
	G.~Masbaum and J.~Roberts.
	\newblock {On central extensions of mapping class groups}.
	\newblock {\em Math. Ann.}, 302:131--150, 1995.
	
	\bibitem[MS89]{ms89}
	G.~Moore and N.~Seiberg.
	\newblock {Classical and Quantum Conformal Field Theory}.
	\newblock {\em Comm. Math. Phys.}, 123:177--254, 1989.
	
	\bibitem[MSWY23]{sn}
	L.~Müller, C.~Schweigert, L.~Woike, and Y.~Yang.
	\newblock The {L}yubashenko modular functor for {D}rinfeld centers via
	non-semisimple string-nets.
	\newblock arXiv:2312.140109 [math.QA], 2023.
	
	\bibitem[MW23]{cyclic}
	L.~Müller and L.~Woike.
	\newblock {Cyclic framed little disks algebras, {G}rothendieck-{V}erdier
		duality and handlebody group representations}.
	\newblock {\em Quart. J. Math.}, 74(1):163--245, 2023.
	
	\bibitem[MW24a]{mwansular}
	L.~Müller and L.~Woike.
	\newblock {Classification of {C}onsistent {S}ystems of {H}andlebody {G}roup
		{R}epresentations}.
	\newblock {\em Int. Math. Res. Not.}, 2024(6):4767--4803, 2024.
	
	\bibitem[MW24b]{mwcenter}
	L.~Müller and L.~Woike.
	\newblock The distinguished invertible object as ribbon dualizing object in the
	{D}rinfeld center.
	\newblock {\em Selecta Math. New Ser.}, 30:98, 2024.
	
	\bibitem[Prz99]{Przytycki}
	J.~Przytycki.
	\newblock {Fundamentals of Kauffman bracket skein modules}.
	\newblock {\em Kobe J. Math.}, 16:45--66, 1999.
	
	\bibitem[Rob94]{roberts}
	J.~Roberts.
	\newblock {Skeins and mapping class groups}.
	\newblock {\em Math. Proc. Camb. Phil. Soc.}, 115:53--77, 1994.
	
	\bibitem[RST24]{rst}
	I.~Runkel, C.~Schweigert, and Y.~H. Tham.
	\newblock Excision for spaces of admissible skeins.
	\newblock arXiv:2407.09302 [math.QA], 2024.
	
	\bibitem[RT90]{rt1}
	N.~Reshetikhin and V.~G. Turaev.
	\newblock {Ribbon graphs and their invariants derived from quantum groups}.
	\newblock {\em Comm. Math. Phys.}, 127:1--26, 1990.
	
	\bibitem[RT91]{rt2}
	N.~Reshetikhin and V.~Turaev.
	\newblock {Invariants of 3-manifolds via link polynomials and quantum groups}.
	\newblock {\em Invent. Math.}, 103:547--598, 1991.
	
	\bibitem[Run20]{Non-s-string-nets}
	I.~Runkel.
	\newblock {String-net models for non-spherical pivotal fusion categories}.
	\newblock {\em J. Knot Theory Ramif.}, 29(6):2050035, 2020.
	
	\bibitem[Seg88]{Segal}
	G.~Segal.
	\newblock {Two-dimensional conformal field theories and modular functors}.
	\newblock In {\em {IX International Conference on Mathematical Physics
			(IAMP)}}, 1988.
	
	\bibitem[Shi19]{shimizumodular}
	K.~Shimizu.
	\newblock {Non-degeneracy conditions for braided finite tensor categories}.
	\newblock {\em Adv. Math.}, 355:106778, 2019.
	
	\bibitem[SS20]{shibatashimizuintegral}
	T.~Shibata and K.~Shimizu.
	\newblock Categorical aspects of cointegrals on quasi-hopf algebras.
	\newblock {\em J. Alg.}, 564:353--411, 2020.
	
	\bibitem[SS21]{shibatashimizu}
	T.~Shibata and K.~Shimizu.
	\newblock {Modified Traces and the Nakayama Functor}.
	\newblock {\em Alg. Rep. Theory}, 2021.
	
	\bibitem[SW21]{dva}
	C.~Schweigert and L.~Woike.
	\newblock {The Hochschild Complex of a Finite Tensor Category}.
	\newblock {\em Alg. Geom. Top.}, 21(7):3689--3734, 2021.
	
	\bibitem[SW23]{tracesw}
	C.~Schweigert and L.~Woike.
	\newblock The trace field theory of a finite tensor category.
	\newblock {\em Alg. Rep. Theory (online first)}, 26:1931--1949, 2023.
	
	\bibitem[Til98]{tillmann}
	U.~Tillmann.
	\newblock {$\mathcal{S}$-Structures for $k$-Linear Categories and the
		Definition of a Modular Functor}.
	\newblock {\em J. London Math. Soc.}, 58(1):208--228, 1998.
	
	\bibitem[Tur90]{turaevck}
	V.~Turaev.
	\newblock {Conway and Kauffman modules of a solid torus}.
	\newblock {\em J. Math. Sci.}, 52:2799--2805, 1990.
	\newblock Translated from Zapiski Nauchnykh Seminarov Leningradskogo Otdeleniya
	Matematicheskogo Instituta im. V. A. Steklova AN SSSR, Vol. 167, pp. 79–89,
	1988.
	
	\bibitem[Tur91]{turaevskein}
	V.~Turaev.
	\newblock {Skein quantization of Poisson algebras of loops on surfaces}.
	\newblock {\em Ann. Sci. Sc. Norm. Sup}, 24(6):635--704, 1991.
	
	\bibitem[Tur10]{turaev}
	V.~G. Turaev.
	\newblock {\em Quantum Invariants of Knots and 3-Manifolds}, volume~18 of {\em
		Studies in Math.}
	\newblock De Gruyter, 2010.
	\newblock 2nd edition (1st edition from 1994).
	
	\bibitem[Wal]{Walker}
	K.~Walker.
	\newblock {TQFT}s.
	\newblock Notes available at \url{http://canyon23.net/math/tc.pdf}.
	
	\bibitem[Woi25]{microcosm}
	L.~Woike.
	\newblock {The cyclic and modular microcosm principle}.
	\newblock {\em Canad. J. Math. (online first)}, 2025.
	
\end{thebibliography}
